\documentclass[titlepage,a4paper,12pt]{amsart} 
\usepackage[foot]{amsaddr}
\usepackage{amssymb}
\usepackage{t1enc}
\usepackage[latin1]{inputenc}
\usepackage[english]{babel}
\usepackage{mathrsfs} 
\usepackage{hyperref}
\usepackage[all]{xy} 
\usepackage[left=3cm,top=3cm,right=3cm,bottom=3cm]{geometry}

\newtheorem{thm}{Theorem}[section]
\newtheorem*{thm*}{Theorem}
\newtheorem{prop}[thm]{Proposition}
\newtheorem{lem}[thm]{Lemma}
\newtheorem{cor}[thm]{Corollary}

\theoremstyle{definition}
\newtheorem{defn}[thm]{Definition}
\newtheorem{rem}[thm]{Remark}
\newtheorem{ex}[thm]{Example}

\DeclareMathOperator\W{\mathcal W}

\DeclareMathOperator\A{\mathcal A}

\DeclareMathOperator\C{\mathbb{C}}
\DeclareMathOperator\R{\mathbb{R}}
\DeclareMathOperator\N{\mathbb{N}}

\DeclareMathOperator\conv{conv}

\DeclareMathOperator\Ad{Ad}

\DeclareMathOperator\ran{Ran}

\DeclareMathOperator\tr{tr}

\DeclareMathOperator\U{U}
\DeclareMathOperator\SU{SU}

\DeclareMathOperator\bd{bd}

\DeclareMathOperator\su{{\mathfrak su}}
\def\u{{\mathfrak u}}

\def\k{{\mathfrak k}}
\def\h{{\mathfrak h}}

\hypersetup{
  colorlinks   = true, 
  urlcolor     = blue, 
  linkcolor    = blue, 
  citecolor   = red    
}
\title{Convexity of sums of eigenvalues of a segment of unitaries}
\date{\today}
\author{Gabriel Larotonda}
\email[Gabriel Larotonda]{glaroton@dm.uba.ar}
\address[Gabriel Larotonda]{Facultad de Ciencias Exactas y Naturales, Universidad de Buenos Aires, Ciudad Universitaria, Avenida Cantilo S/N (1428), Buenos Aires, Argentina}

\author{Martin Miglioli}
\email[Martin Miglioli]{martin.miglioli@gmail.com}
\address[Martin Miglioli]{Instituto Argentino de Matem\'atica-CONICET. Saavedra 15, Piso 3, (1083) Buenos Aires, Argentina}



\begin{document}

\begin{abstract}
For a $n\times n$ unitary matrix $u=e^z$ with $z$ skew-Hermitian, the angles of $u$ are the arguments of its spectrum, i.e. the spectrum of $-iz$. For $1\le m\le n$, we show that $s_m(t)$, the sum of the first $m$ angles of the path $t\mapsto e^{tx}e^y$ of unitary matrices, is a convex function of $t$ (provided the path stays in a vecinity of the identity matrix). This vecinity is described in terms of the opertor norm of matrices, and it is optimal. We show that the when all the maps $t\mapsto s_m(t)$ are linear, then $x$ commutes with $y$. Several application to unitarily invariant norms in the unitary group are given. Then we extend these applications to $\Ad$-invariant Finsler norms in the special unitary group of matrices. This last result is obtained by proving that any $\Ad$-invariant Finsler norm in a compact semi-simple Lie group $K$ is the supremum of a family of what we call orbit norms, induced by the Killing form of $K$. 
\end{abstract}

\maketitle

\tableofcontents



\section{Introduction}

In this paper we study the eigenvalues of a geodesic segment in the unitary group $\U(n)$ of matrices 
\[
 t\mapsto u(t)= e^{itx}e^{iy} \qquad t\in\mathbb R
\]
where $x,y$ are self-adjoint $n\times n$ complex matrices, and in particular we establish that
\[
t\mapsto \|\log(e^{itx}e^{iy})\|_{\infty}
\]
is convex as long as $\|u(t)-1\|_{\infty}<\sqrt{2}$. This eigenvalue problem  has been approached with several techniques in the literature, in different contexts (matrices, compact operators, bounded operators, etc), see \cite{pschatten, fredholm, finitemeasure} and also \cite{miglioli} and the references therein. It is part of what in the literature is known loosely as product eigenvalue problems, and connected to many relevant results in matrix analysis and operator theory, such as the quantum Horn conjecture/theorem proved by Belkale \cite{belkale} which characterizes the eigenvalues of a product of unitary matrices.  For techniques more closely related to those in this paper, see also \cite{miglioli24} and the references therein.  With indirect relation to our results, convexity properties of the  condition number of a matrix are known, see for instance \cite{beltran} and the references therein.

In particular, this convexity allows for a good control in the  dispersion of the spectrum of a path of unitary matrices. More to the point, since the neighbourhood of convexity we obtain with our techniques in this paper is optimal, we are able to use our results to obtain the optimal radius of convexity of balls in the unitary group $\U(n)$, for distances that are bi-invariant for the action of the group itself. Our results can be stated in terms of the distance map $t\mapsto d(t)$ from a point to a geodesic segment in the unitary group of matrices, and it can be though of as a generalization of a well-known result in the Riemannian geometry of the unit sphere  $S^n\subset \mathbb R^{n+1}$, that says that such a distance map from the nort pole of the sphere to a geodesic segment is convex, as long as the geodesic stays in the upper hemisphere.


\medskip

What follows is a brief description of the organisation and main results of this paper. In Section 2 introduce the notations and recall (with a self-included proof) the first and second variation formulas for the eigenvalues of a smooth path of normal matrices. In Section 3 we specialize these formulas to our segment $t\mapsto u(t)=e^{itx}e^{iy}$, with eigenvalues $e^{i\theta_k(t)}$ and $\theta_1(t)>\theta_2(t)>\dots >\theta_n(t)$, and with this we prove that for each $1\le m\le n$, the map $t\mapsto \sum_{i=1}^m\theta_i(t)$ is convex provided $\|u(t)-1\|_{\infty}<\sqrt{2}$ (the norm here is the spectral norm). Later it is proved that this neighbourhood is optimal. The case  of linear maps is proved to be equivalent to the commuation of $x$ and $y$. Then by means of a limiting argument, we extend the result to the case of collapsing eigenvalues $\theta_i$, which is one of the main result of this paper (Theorem \ref{convsumdistgeneral}). Similar results are obtained for the partial sums of singular values. In Section 4, we use these results to prove that for any unitarily invariant norm $\|\cdot\|$, if we write $e^{ix(t)}=u(t)=e^{itx}e^{iy}$ for $x(t)$ the skew-adjoint logarithm of $u$, then $t\mapsto \|x(t)\|$ is convex, provided $\|x(t)\|_{\infty}<\frac{\pi}{2}$ (Theorem \ref{convdistancNUI}).  In Section 5 we move on to application to \textit{weakly invariant norms} i.e. norms invariant for the adjoint action of the group on itself, in particular for $\SU(n)$ we only ask that $\|UXU^*\|=\|X\|$ for any $X\in \su(n)$ and any $U\in \SU(n)$. The case of (only) positively homogeneous norms is also included in this discussion, in what we call $\Ad$\textit{-invariant Finsler norms}. Ellaborating on some ideas that follow from Kostant's theorem, we prove (Theorem \ref{normsuphofer}) that for a compact semi-simple group $K$ with Lie algebra $\mathfrak k$, any $\Ad$-invariant Finsler norm can be obtained as a supremum of a distinguished family of norms induced by the Killing form (what we  call \textit{orbit norms}). This extends a known result of Horn and Mathias \cite{horn} to this general setting. Combining this and the previous results, we extend the convexity theorem to weakly invariant norms in $\SU(n)$, and we characterize the case of strict convexity for orbit norms with different eigenvalues (Theorem \ref{striconv}). In Section 6 we finish the paper with applications of the  previous results, combined with our results from \cite{larmi}, to establis the convexity of the map 
\[
t\mapsto d_{\phi}(1,e^{itx}e^{iy})
\]
where $d_{\phi}$ is the rectifiable distance induced in $\U(n)$ or $\SU(n)$ by an invariant norm (Theorems \ref{convadsu} and \ref{convrecu}).

\section{First and second variation formulas for eigenvalues}

Let $I\subset \mathbb R$ be a compact connected interval, and  let $u:I\to M_n(\mathbb C)$ be a smooth path of normal matrices, and let $\lambda_k(t)$ be its eigenvalues. Assume that all of them are distinct for each $t\in I$. Then, being the zeroes of the characteristic polynomial $p(t,\lambda)=0$ (and since there are no multiple roots) the map $p$ is regular and 
$$
\frac{\partial p}{\partial \lambda}(t,\lambda_k(t))\ne 0.
$$
By  means of the implicit function theorem, it follows that $t\mapsto \lambda_k(t)$ is as smooth as $t\mapsto u(t)$. For each $k$, we pick a unit norm eigenvector $v_k(t)$ for $\lambda_k(t)$, i.e.
\begin{equation}\label{uvlv}
u(t)v_k(t)=\lambda_k(t)v_k(t).
\end{equation}
The functions $v_k$, being a solution of a linear system with smooth coefficients, can be chosen smooth. A practical way of doing this is to select a circle $C_k\subset \mathbb C$ containing each eigenvalue, such that it contains $\lambda_k(t)$ for all $t\in I=[0,b]$ (since $\lambda_k$ are continuous, this can be done for a compact interval which we assumme is the whole of $I$). Now for each $k=1,\cdots,n$ we consider the eigenprojection
$$
p_k(t)=\frac{1}{2\pi i}\oint_{C_k} (\lambda- u(t))^{-1} d\lambda.
$$
Clearly $p_k:I\to M_n(\mathbb C)_{sim}$ is smooth (as smooth as $u$). Now let 
$$
\varepsilon_k(t)=2p_k(t)-1,
$$
which is also smooth and is a symmetry with fixed subspace the range of $p_k(t)$. As long as $\|p_k(t)-p_k(0)\|<1$, we see that $\|\varepsilon_k(t)\varepsilon_k(0)-1\|<2$. Then it is well-known that if we set
$$
z_k(t)=\frac{1}{2}\ln(\varepsilon_k(t)\varepsilon_k(0))
$$
then $z_k(t)$ is a direct rotation that exchanges the symmetries i.e. $e^{2z_k(t)}\varepsilon_k(0)=\varepsilon_k(t)$ (this can be checked by hand). Now for $t=0$ we pick a unit norm vector $v_k\in \ran(p_k(0))$, and we consider the smooth paths of unit vectors  
$$
v_k(t)=e^{z_k(t)}v_k.
$$
Then it can be readily checked that these are eigenvectors for $u$, i.e. (\ref{uvlv}) is satisfied for each $k$.

Note that since $u$ is normal, eigenvectors corresponding to different eigenvalues are orthogonal, so we a have a smooth, moving orthonormal basis $\{v_k(t)\}_{k=1,\dots,n}$ spanning the whole space. 

What follows are the classical first and second variation formulas for the eigenvalues of $u$; we include a proof since it is elementary and we want this article to be as complete as possible. We use $\langle  v,w\rangle$ to denote the standard inner product in $\mathbb C^n$, which is conjugate-linear in its second entry.

\begin{thm}\label{variaa} Let $u:I\to M_n(\mathbb C)$ be a $C^r$ path of normal matrices ($r\ge 1$). Assumme that all the eigenvalues $\lambda_k(t)$ of $u(t)$ are different for all $t$. Let $\{v_k(t)\}_{k=1,\dots,n}$ be a smooth orthonormal basis of its eigenvectors. Then the first variation formula is 
$$
\lambda_k'(t)=\langle u'(t)v_k(t), v_k(t)\rangle  
$$
and the second variation formula is
$$
\lambda_k''(t)= \langle u''(t) v_k(t), v_k(t)\rangle + 2 \sum_{j\ne k} \frac{\langle u'(t)v_k(t), v_j(t)\rangle \langle u'(t) v_j(t),v_k(t)\rangle}{\lambda_k(t)-\lambda_j(t)}.
$$
\end{thm}
\begin{proof}
From $uv_k=\lambda_k v_k$, we have $u'v_k+uv_k'=\lambda_k'v_k+\lambda_k v_k'$, and multiplying by $v_j^*$ we get 
$$
\langle u'v_k,v_j\rangle +\langle u v_k',v_j\rangle =\lambda_k' \delta_{jk}+ \lambda_k \langle v_k',v_j\rangle
$$
Now 
$$
\langle u v_k',v_j\rangle= \langle v_k',u^*v_j\rangle=\langle  v_k',\overline{\lambda_j}v_j\rangle=\lambda_j  \langle v_k',v_j\rangle,
$$
hence
$$
\langle u'v_k,v_j\rangle  =\lambda_k' \delta_{jk}+ (\lambda_k-\lambda_j) \langle v_k',v_j\rangle.
$$
In particular taking $j=k$ we obtain the first variation formula. Now note that for each $k$, we have $v_k'=\sum_{j=1}^n c_j^k v_j$ for the smooth functions $c_j=\langle v_k',v_j\rangle$. Then from the previous identity we see that
$$
c_j^k=\frac{\langle u'v_k,v_j\rangle}{\lambda_k-\lambda_j}  \quad  \mbox{ for }\quad j\ne k,
$$
hence
\begin{equation}\label{vprima}
v_k'= \sum_{j\ne k}\frac{\langle u'v_k,v_j\rangle}{\lambda_k-\lambda_j} v_j +  c_k^k v_k.
\end{equation}
Now we differentiate the first variation formula:
\begin{equation}\label{sec}
\langle u'' v_k,v_k\rangle +\langle u'v_k',v_k\rangle +\langle u' v_k,v_k'\rangle = \lambda_k'',
\end{equation}
and using (\ref{vprima}) we compute
$$
\langle u'v_k',v_k\rangle=\sum_{j\ne k}\frac{\langle u'v_k,v_j\rangle}{\lambda_k-\lambda_j}\langle u' v_j,v_k\rangle  +  c_k^k  \langle u'v_k,v_k\rangle
$$
and
$$
\langle u' v_k,v_k'\rangle= \sum_{j\ne k}\frac{\langle u'v_k,v_j\rangle}{\lambda_k-\lambda_j}\langle u' v_k,v_k\rangle  +  \overline{c_k^k}  \langle u'v_k,v_k\rangle
$$
Now from $\|v_k\|=1$ we see that 
$$
2\textrm{Re}\langle v_k',v_k\rangle=\langle v_k',v_k\rangle+\langle v_k,v_k'\rangle= \frac{d}{dt}\|v_k\|^2=0,
$$
hence  $c_k^k=\langle v_k',v_k\rangle$ is purely imaginary. Thus adding the previous two expressions, we see from (\ref{sec}) that
$$
\lambda_k''=\langle u'' v_k,v_k\rangle +2\sum_{j\ne k}\frac{\langle u'v_k,v_j\rangle\langle u' v_j,v_k\rangle}{\lambda_k-\lambda_j},
$$
which is the second variation formula.
\end{proof}

\section{Convexity of geodesic segments in $U_n(\mathbb C)$}

\subsection{Second variational formula for the arguments of eigenvalues}

A path $u(t)=e^{itx}e^{iy}$ of unitaries, $x^*=x$, $y^*=y$, is called a segment since it is a geodesic arc for the canonical Lie group connection induced by any bi-invariant metric in the unitary group $U_n(\mathbb C)$. 

We assumme that the path is subject to the condition $\|u(t)-1\|<2$. If $iz(t)=\log(u(t))$ is the (smooth) principal branch of its logarithm, then $z^*=z$, $\|z\|<\pi$. We are interested in the eigenvalues of such path. 

To this end, let $-\pi<\theta_1(t)\leq\theta_2(t)\leq\dots\leq\theta_n(t)<\pi$ be the eigenvalues of $z(t)$, i.e. $\lambda_k(t)=e^{i\theta_k(t)}$ are the eigenvalues of $u(t)$. If we assume that the eigenvalues of $u(t)$ are distinct for all $t$, then each $\theta_k:I \to \R$ is smooth from the previous discussion. As before, let  $v_k(t)$ be the corresponding unit norm eigenvector, that is $u(t)v_k(t)=e^{i\theta_k(t)}v_k(t)$.

\begin{prop}\label{propconvsum}
Let $-\pi<\theta_1(t)<\theta_2(t)<\dots<\theta_n(t)<\pi$ be the arguments of the eigenvalues of a path $u(t)=e^{itx}e^{iy}$ such that $\|u(t)-1\|<2$ for all $t\in [a,b]$. Then, for $k=1,\dots,n$ we have 
\begin{equation}\label{thetak}
\theta_k'(t)=\langle xv_k(t),v_k(t)\rangle,
\end{equation}
\begin{align}\label{theta2}
\theta_k''(t) &=  2 \sum_{j\ne k} \frac{ \sin(\theta_j(t)-\theta_k(t))}{|e^{i\theta_k(t)}-e^{i\theta_j(t)}|^2}|\langle x v_k(t), v_j(t)\rangle|^2.
\end{align}
and for each $1\le m\le n$
\begin{equation}\label{theta3}
\sum_{k=1}^m\theta_k''(t)= 2 \sum_{k=1}^m\sum_{j=m+1}^n \frac{ \sin(\theta_k(t)-\theta_j(t))}{|e^{i\theta_k(t)}-e^{i\theta_j(t)}|^2}|\langle x v_k(t), v_j(t)\rangle|^2.
\end{equation}
\end{prop}

\begin{proof}
In our case $u'(t)=ix u(t)$, hence the first variation formula tells us that $
\theta_k'(t)=\langle xv_k(t),v_k(t)\rangle$. Since $u''(t)=-x^2u(t)$, the second variation formula tells us that
$$
\lambda_k''(t)= -e^{i\theta_k(t)}\langle x^2v_k(t), v_k(t)\rangle   - 2 \sum_{j\ne k} e^{i(\theta_k(t)+\theta_j(t))}\frac{|\langle x v_k(t), v_j(t)\rangle|^2}{e^{i\theta_k(t)}-e^{i\theta_j(t)}} .
$$
Comparing with $\lambda_k''(t)= e^{i\theta_k(t)}(-\theta_k'(t)^2+i\theta_k''(t))$,  we get
$$
-\langle x^2v_k(t), v_k(t)\rangle   - 2 \sum_{j\ne k} e^{i\theta_j(t)}\frac{|\langle x v_k(t), v_j(t)\rangle|^2}{e^{i\theta_k(t)}-e^{i\theta_j(t)}}=  -\theta_k'(t)^2+i\theta_k''(t).
$$
Since $x^*=x$, the term $\langle x^2 v_k,v_k\rangle$ is real, and taking the real part in the previous equation we recover (\ref{thetak}), but with a modulus. But if we take the imaginary parts, we see that
\begin{align*}
\theta_k''(t) & = -2 \sum_{j\ne k} \textrm{Im}\{\frac{e^{i\theta_j(t)}}{e^{i\theta_k(t)}-e^{i\theta_j(t)}}\}|\langle x v_k(t), v_j(t)\rangle|^2 \nonumber\\
\end{align*}
therefore
$$
\theta_k''(t)= 2 \sum_{j\ne k} \frac{ \sin(\theta_j(t)-\theta_k(t))}{|e^{i\theta_k(t)}-e^{i\theta_j(t)}|^2}|\langle x v_k(t), v_j(t)\rangle|^2.
$$
Now note that
\begin{align*}
\sum_{k=1}^m\theta_k''(t)=&2 \sum_{k=1}^m\sum_{j\ne k} \frac{ \sin(\theta_k(t)-\theta_j(t))}{|e^{i\theta_k(t)}-e^{i\theta_j(t)}|^2}|\langle x v_k(t), v_j(t)\rangle|^2\\
=&S_1(t)+S_2(t)
\end{align*}
where 
$$S_1(t)=2 \sum_{(j,k): 1\leq j,k\leq m,j \neq k} \frac{ \sin(\theta_k(t)-\theta_j(t))}{|e^{i\theta_k(t)}-e^{i\theta_j(t)}|^2}|\langle x v_k(t), v_j(t)\rangle|^2$$
and
$$
S_2(t)=2 \sum_{k=1}^m\sum_{j=m+1}^n \frac{ \sin(\theta_k(t)-\theta_j(t))}{|e^{i\theta_k(t)}-e^{i\theta_j(t)}|^2}|\langle x v_k(t), v_j(t)\rangle|^2.
$$
In the sum $S_1(t)$, for each $\{j,k\}$ such that $j\neq k$ the indices $(j,k)$ and $(k,j)$ appear in the sum. Note that

\begin{align*}
\frac{ \sin(\theta_k(t)-\theta_j(t))}{|e^{i\theta_k(t)}-e^{i\theta_j(t)}|^2}|\langle x v_k(t), v_j(t)\rangle|^2&=-\frac{ \sin(\theta_j(t)-\theta_k(t))}{|e^{i\theta_k(t)}-e^{i\theta_j(t)}|^2}|\langle x v_k(t), v_j(t)\rangle|^2\\
&=-\frac{ \sin(\theta_j(t)-\theta_k(t))}{|e^{i\theta_j(t)}-e^{i\theta_k(t)}|^2}|\langle x v_j(t), v_k(t)\rangle|^2,
\end{align*}
so that 
$$\frac{ \sin(\theta_k(t)-\theta_j(t))}{|e^{i\theta_k(t)}-e^{i\theta_j(t)}|^2}|\langle x v_k(t), v_j(t)\rangle|^2+\frac{ \sin(\theta_j(t)-\theta_k(t))}{|e^{i\theta_j(t)}-e^{i\theta_k(t)}|^2}|\langle x v_j(t), v_k(t)\rangle|^2=0.$$
Therefore the $S_1(t)=0$. 
\end{proof}

\begin{rem}\label{l1} Since $e^{Tr A}=\det(e^A)$ for any $A\in M_n(\mathbb C)$,  for $x,y$ as in Theorem \ref{variaa} we have
$$
e^{i \sum_k \theta_k(t)}=e^{Tr(\log(u(t)))}=\det(e^{\log(u(t))})=\det(e^{itx}e^y)=e^{it Tr(x)}e^{i Tr(y)}.
$$
Differentiating we see that $e^{i \sum_k \theta_k(t)}(i\sum_k \theta_k'(t))=e^{iTr(y)}e^{it Tr(x)}i \, Tr(x)$. Hence cancelling we see that 
$$
\sum_{k=1}^n \theta_k'(t)=Tr(x),
$$
and since $u(0)=e^{iy}$ it must be
$$
\sum_{k=1}^n  \theta_k(t)=t	\, Tr(x)+ Tr(y).
$$
\end{rem}

\subsection{Segments with distinct eigenvalues}

In what follows we show that the partial increasing sums of these angles are always convex mappings.

\begin{thm}\label{convsumadistint}
If $\|e^{itx}e^{iy}-1\|< \sqrt{2}$ and $u(t)=e^{itx}e^{iy}$ has distinct eigenvalues $e^{i\theta_k(t)}$ with $\theta_1(t)>\theta_2(t)>\dots>\theta_n(t)$ for all $t\in [0,t_0]$, then for $1\leq m\leq n$ 
$$t\mapsto \sum_{i=1}^m\theta_i(t)$$ 
is convex in $[0,t_0]$.
\end{thm}

\begin{proof}
In \eqref{theta3} the summands are  
$$ \frac{ \sin(\theta_k(t)-\theta_j(t))}{|e^{i\theta_k(t)}-e^{i\theta_j(t)}|^2}|\langle x v_k(t), v_j(t)\rangle|^2$$
with $k<j$. Note that $\|u(t)-1\|\le \sqrt{2}$ is equivalent $-\pi/2\le \theta_i(t)\le \pi/2$ for all $i$, thus $0\le\theta_k(t)-\theta_j(t)\le \pi$ for all $j$, and this implies $\sin(\theta_k(t)-\theta_j(t))\geq 0$. Therefore the summands are nonnegative, so the second derivative is non-negative.
\end{proof}

\subsection{The condition $\sum\theta_k''(t)=0$}

\begin{rem}If $\|e^{itx}e^{iy}-1\|\le \sqrt{2}$ for $t\in [0,t_0]$ then it must be $\|y\|\le \frac{\pi}{2}$ and $\|tx\|< \pi$ for all $t\in [0,t_0]$. In fact, at $t=0$ we have $\|e^{iy}-1\|\le\sqrt{2}$ which bounds the norm of $y$, and then by the triangle inequality
$$
\|e^{itx}-1\|\le \|e^{itx}-e^{-iy}\|+\|e^{-iy}-1\|\le 2\sqrt{2}.
$$
Now if $|e^{i\theta}-1|\le 2\sqrt{2}$, an elementary computation shows that $|\cos(\theta)|\le \sqrt{2}-1$, which in turn implies that $|\theta|\le 0,87\pi$.
\end{rem}

Let $x,y$ be as in Theorem \ref{convsumadistint}. In particular we have 
\begin{equation}\label{segundaa}
\theta_1''(t)  = 2 \sum_{j=2}^n  \frac{ \sin(\theta_1(t)-\theta_j(t))}{|e^{i\theta_1(t)}-e^{i\theta_j(t)}|^2}|\langle x v_1(t), v_j(t)\rangle|^2.
\end{equation}
Then
\begin{lem}\label{caso1}
For fixed $t\in [0,t_0]$, the following are equivalent: 
\begin{enumerate}
\item[(a)] $\theta_1''(t)=0$
\item[(b)] $xv_1(t)=\theta_1'(t)v_1(t)$
\item[(c)] $yv_1(t)=(\theta_1(t)-t\theta_1'(t))v_1(t)$.
\end{enumerate}
In particular $v_1(t)$ is a simultaneous eigenvector of $x$ and $y$.
\end{lem}
\begin{proof}
If the eigenvalues do not interlace, we have $\sin(\theta_1(t)-\theta_j(t))>0$ for $j>1$. The case of $\theta_1''(t)=0$ is then only possible by (\ref{segundaa}) if $\langle xv_1,v_j\rangle=0$ for all $j\ge 2$. But in this case it must be $xv_1(t)= \mu(t)v_1(t)$ for some real analytic function $\mu$. Plugging this in the first variation formula, for $k=1$, we see that $\mu(t)=\theta_1'(t)$, that is $xv_1(t)=\theta_1'(t)v_1(t)$, which proves $(b)$. Now assumme that $(b)$ holds, from $uv_1=\lambda_1v_1$, we get
\begin{align*}
0& = e^{i\theta_j(t)}\langle v_j(t),v_1(t)\rangle =\langle e^{itx}e^{iy} v_j(t),v_1(t)\rangle=\langle e^{iy} v_j(t),e^{-itx}v_1(t)\rangle\\
&=\langle e^{iy} v_j(t),e^{-it\theta_1'(t)}v_1(t)\rangle=e^{i t\theta_1'(t)}\langle e^{iy} v_j(t),v_1(t)\rangle.
\end{align*}
Thus $\langle v_j(t),e^{-iy}v_1(t)\rangle=0$ for all $j> 1$, and this is only possible if $e^{-iy}v_1(t)=\beta(t)v_1(t)$; note that $|\beta(t)|=1$. Now
$$
e^{i\theta_1(t)}v_1(t)=e^{itx}e^{iy}v_1(t)=e^{itx}\overline{\beta(t)}v_1(t)=\overline{\beta(t)}e^{it\theta_1'(t)}v_1(t),
$$
which implies that $e^{iy}v_1(t)=e^{i(\theta_1(t)-t\theta_1'(t))}v_1(t)$. Note that since at $t=0$ we have $\|e^{iy}-1\|\le \sqrt{2}$, it must be  $\|y\|\le\frac{\pi}{2}$, hence there exists $k\in\mathbb Z$ such that $yv_1(t)=(\theta_1(t)-t\theta_1'(t)+2k\pi)v_1(t)$. Hence $(y+tx)v_1(t)=(\theta_1(t)+2k\pi)v_1(t)$, and taking norms we see that
$$
|\theta_1(t)+2k\pi|=\|y+tx\|\le  t\|x\|+ \frac{\pi}{2}<\frac{3\pi}{2}
$$
by the previous remark. Since $-\frac{\pi}{2}\le \theta_1(t)\le  \frac{\pi}{2}$, this is only possible if $k=0$. Then $(y+tx)v_1(t)=\theta_1(t)v_1(t)$ which proves $(c)$. Now if $(c)$ holds, then
$$
e^{i\theta_1(t)}v_1(t)=e^{itx}e^{iy}v_1(t)=e^{itx}e^{i\theta_1(t)}e^{-it\theta_1'(t)}v_1(t),
$$
hence $e^{itx}v_1(t)=e^{it\theta_1'(t)}v_1(t)$. Then, since $\|tx\|<\pi$, by the first variation formula (\ref{thetak}) we also have $|t\theta_1'(t)|<\pi$, which implies that $xv_1(t)=\theta_1'(t)v_1(t)$. This certainly implies $(a)$ by (\ref{segundaa}).
\end{proof}

\begin{prop}Let $x,y$ be as in Theorem \ref{convsumadistint}. If $\sum_{k=1}^m\theta_k(s)''=0$ for some $t\in [0,t_0]$, and all $1\le m\le n-1$, then the span of $\{v_1(s),\cdots,v_m(s)\}$ is invariant for both $x$ and $y$, and $[x,y]=0$.
\end{prop}
\begin{proof}
Since $\theta_k(t)-\theta_j(t)>0$ for $j>k$, in equation (\ref{theta3}) all the terms must vanish, then is easy to check that for each $k=1,\dots m$ (following the proof of the previous lemma) that it must be $\langle xv_k(s),v_j(s)\rangle=0$ for all $j\ne k$, $j=1,\dots n$.  Then we have a common basis for diagonalizing $x,y$ therefore $[x,y]=0$.
\end{proof}

\subsection{Convexity, general case}
\medskip

We now prove that segments with repeated eigenvalues can be approximated by segments with different eigenvalues for all $t$.

\begin{prop}\label{aproxdistint} 
Let $u(t)=e^{itx}e^{iy}$ satisfy $\|u(t)-1\|< r<2$ for all $t\in (a,b)$. Then for each $t\in (a,b)$ there exists $y_n\to y$ such that for each $n$ the segment $u_n(t)=e^{itx}e^{iy_n}$ satisfies $\|u_n(t)-1\|<r$ and has distinct eigenvalues in some open interval $I\subset (a,b)$ around $t$. 
\end{prop}

\begin{proof}
Denote 
$$A=\{z:z=z^*\mbox{ and }z\perp x\}.$$
We define the map $\psi:\R\times A\to U_n(\mathbb C)$ as
$$\psi(t,z)=e^{itx}e^{iz}e^{iy}.$$
The differential at a point $(t,0)$ with $t\in\R$ is 
$$
\psi_{*(t,0)}(s,z)=e^{itx}isxe^{iy}+e^{itx}ize^{iy}=ie^{itx}(sx
+z)e^{iy}.
$$
We note that since $(s,z)\mapsto z+sx$ is an isomorphism from $\mathbb R\times A$ onto $T_{e^{tx}e^y}U_n(\mathbb C)\simeq A\oplus \{x\}$ the map $\psi$ is a local diffeomorphism around $(t,0)$ with $t\in (a,b)$: there exists an $\epsilon>0$ such that if $I=(t-\varepsilon,t+\varepsilon)$ and $A_{\epsilon}=A\cap B(0,\epsilon)$, then 
$$\psi_{\epsilon}=\psi\vert_{I\times A_{\epsilon}}$$
is a diffeormorphism onto its image and moreover  $\|\psi_{\epsilon}(t,z)-1\|< r$ for all $(t,z)\in  I\times A_{\epsilon}$ (since $\psi(t,0)=e^{itx}e^{iy})$.  From now we denote this restricted map with $\psi$. Now, it is well known that the set of Hermitian matrices with at least one repeated eigenvalue has codimension $3$ inside the Hermitian matrices \cite[p.142]{laxla}. 

Then the assertion is also true for the the open ball $\|x\|<\frac{\pi}{2}$ inside the Hermitian matrices. This ball is diffeomorphic with the (relatively) open set
$$
U_n(\mathbb C)\cap\{u: \|u-1\|<2\}=:U_0\subset U_n(\mathbb C)
$$
by means of the exponential map $x\mapsto e^{ix}$. Then it follows that the set $S$ of unitary matrices in $U_0$ with repeated eigenvalues has codimension 3 in $U_0$.

Therefore $\psi^{-1}(S\cap U_0)$ has codimension at least $3$ in $I\times A_{\epsilon}$. We denote by 
$$
p:I\times A_{\epsilon}\to A_{\epsilon}
$$ 
the projection onto the second factor and note that $p(\psi^{-1}(S\cap U_n(\mathbb C))$ has codimension at least $2$ in $A_{\epsilon}$, $p$ being a smooth submersion with one dimensional kernel. Therefore, there is $z_n\to 0$ such that $z_n\in A_{\epsilon}$ and such that $z_n\notin p(\psi^{-1}(S\cap U))$, which means that $\psi(t,z_n)=e^{itx}e^{iz_n}e^{iy}$ has distinct eigenvalues for all $t\in I$. If we set $y_n$ such that $e^{iy_n}=e^{iz_n}e^{iy}$ then the curve $u_n(t)=e^{itx}e^{iy_n}$ satisfies the conclusions of the proposition.  
\end{proof}

\bigskip

\subsection{Doubling the spectrum}

Let $J$ be a complex conjugation on $\C^n$ which we denote with $\overline{u}=JuJ$. Note that if $(e_i)_{i=1}^n$ is an orthonormal basis of $\C^n$ then a complex conjugation is defined by $J:\sum_{i=1}^n\alpha_ie_i\mapsto\overline{\alpha_i}e_i$. In this case the matrix of $\overline{u}$ is the complex conjugate of the matrix of $u$. The representation $u\mapsto JuJ$ is isomorphic to the dual of the standard representation $u\mapsto u$. We consider the representation of $U_n(\mathbb C)$ on $\C^{2n}$ given by
$$\rho(u)=u\oplus\overline{u}.$$

Recall that the \textit{singular values} $\{\sigma_i(x)\}_{i=1,\dots,n}$ of $x\in M_n(\mathbb C)$ are the eigenvalues of $|x|=\sqrt{x^*x}\ge 0$, ordered in a non-increasing fashion.

\begin{lem}\label{lemaconjug}
Let $u\in U_n(\mathbb C)$ with eigenvalues $e^{i\theta_1},\dots,e^{i\theta_n}$ and let $\sigma_1\geq \sigma_2\geq\dots\geq\sigma_n\geq 0$ be the singular values of 
$$-i\log(u).$$
Then the eigenvalues of $\overline{u}$ are  $e^{-i\theta_1},\dots,e^{-i\theta_n}$ and the eigenvalues of $\rho(u)=u\oplus\overline{u}$ are $e^{i\sigma_1},\dots,e^{i\sigma_n},e^{-i\sigma_n},\dots,e^{-i\sigma_1}$ with $\sigma_1\geq\sigma_2\geq\dots\sigma_n\geq-\sigma_n\geq\dots\geq -\sigma_1$.
\end{lem}

\begin{proof}
The first assertion follows from $J\lambda uJ=\overline{\lambda}JuJ$ for $\lambda\in\C$. If the eigenvalues of $u$ are $e^{i\theta_1},\dots,e^{i\theta_n}$ we define $\theta_j=-\theta_{j-n}$ for $j=n+1,\dots,2n$. Then the eigenvalues of $\rho(u)=u\oplus \overline{u}$ are $e^{i\theta_1},\dots,e^{i\theta_{2n}}$. The lemma follows. 
\end{proof}

\begin{thm}\label{convsumdistgeneral}
Let $u(t)=e^{itx}e^{iy}$ with $\|u(t)-1\|< \sqrt{2}$ in $[0,t_0]$ have eigenvalues $e^{i\theta_k(t)}$ with $\theta_1(t)\geq\theta_2(t)\geq\dots\geq\theta_n(t)$, and denote $\sigma_1(t)\geq\sigma_2(t)\geq\dots\geq\sigma_n(t)\ge 0$ the singular values of $-i\log(u)$. Then for $1\leq m\leq n$ 
$$t\mapsto\sum_{i=1}^m\theta_i(t)$$ 
is convex in $[0,t_0]$,
$$t\mapsto\sum_{i=m}^n\theta_i(t)$$
is concave in $[0,t_0]$, and
$$t\mapsto\sum_{i=1}^m\sigma_i(t)$$
is convex in $[0,t_0]$,
\end{thm}

\begin{proof}
Being a local property, it suffices to check the assertions around each $t_1\in  [0,t_0]$. By the previous theorem there is $y_j\to y$ such that for each $j$ the segment $u_j(t)=e^{itx}e^{iy_j}$ satisfies $\|u_j(t)-1\|< \sqrt{2}$ and has distinct eigenvalues in some interval $I$ around $t_1$. We denote with $\theta_1(t)\geq\theta_2(t)\geq\dots\geq\theta_n(t)$ the arguments of the eigenvalues of $u(t)$ and with $\theta_{1,j}(t)>\theta_{2,j}(t)>\dots>\theta_{n,j}(t)$ the arguments of the eigenvalues of $u_j(t)$. Then
$$\theta_k(t)=\lambda_k(-i\log(u(t)))$$
and
$$\theta_{k,j}(t)=\lambda_k(-i\log(u_j(t))).$$
Note that 
\begin{align*}
\left\vert\sum_{k=1}^m\theta_k(t)-\sum_{k=1}^m\theta_{k,j}(t)\right\vert&\leq\sum_{k=1}^m\left\vert\theta_k(t)-\theta_{k,j}(t)\right\vert\leq \sum_{k=1}^n\left\vert\theta_k(t)-\theta_{k,j}(t)\right\vert\\ 
&=\Vert(\theta_k(t)-\theta_{k,j}(t))_{k=1}^n\Vert_{l^1}\\
&=\Vert(\lambda_k(-i\log(u(t)))-\lambda_k(-i\log(u_j(t)))_{k=1}^n\Vert_{l^1}\\
&\leq \left\Vert -i(\log(u(t))-\log(u_j(t)))\right\Vert_1
\end{align*}
where the last inequality follows from the Hoffman-Lidskii-Wielandt inequality  \cite[Theorem 9.7]{bha}. For fixed $t\in I$ we have $e^{itx}e^{iy_j}\to e^{itx}e^{iy}$ as $j\to\infty$ since $y_j\to y$. Therefore, $\log(u_j(t))\to\log(u(t))$ and we conclude that for all $t\in I$  
$$
\sum_{k=1}^m\theta_{k,j}(t)\to\sum_{k=1}^m\theta_k(t).
$$
The functions $\sum_{k=1}^m\theta_{k,j}(t)$ are convex in $I$ for each $j$ by Theorem \ref{convsumadistint} and converge pointwise to the function $\sum_{k=1}^m\theta_k(t)$ there, hence $\sum_{k=1}^m\theta_k(t)$ is convex in $I$.

From the first assertion, by means of  Lemma \ref{lemaconjug}, we obtain the second assertion of this lemma if we take the curve $\overline{u(t)}=e^{-it\overline{x}}e^{-i\overline{y}}$. The last statement of this lemma follows from Lemma \ref{lemaconjug} if we take the curve 
$$\rho(u(t))=e^{it(x\oplus-\overline{x})}e^{i(y\oplus-\overline{y})}.$$

\end{proof}

\begin{rem}
The $\|u(t)-1\|< \sqrt{2}$ bound in the previous theorem is optimal, see \cite[Example 3.11]{miglioli} for an example in $n=2$.
\end{rem}

\bigskip

\section{Unitarily invariant norms in $U_n(\mathbb C)$}

For $n\in\N$ consider
$$\R^n_{+,\downarrow}=\{(\alpha_1,\dots,\alpha_n)\in\R^n:\alpha_1\geq \alpha_2\geq\dots\geq \alpha_n\geq 0\mbox{  and  }\alpha_1>0\}.$$
For an integer $1\leq k\leq n$ the Ky-Fan norms are defined by
$$\|x\|_{(k)}=\sum_{i=1}^k\sigma_i(x),
$$
where $\sigma_1(x)\ge \sigma_2(x)\ge \dots \ge \sigma_n(x)\ge 0$ is the string of singular values of $x$. For $\alpha\in \R^n_{+,\downarrow}$ we define 
$$\|x\|_{\alpha}=\sum_{i=1}^n\alpha_i\sigma_i(x).$$
Note that by setting $\alpha_{n+1}=0$ we have
$$\|x\|_{\alpha}=\sum_{i=1}^n(\alpha_i-\alpha_{i+1})\|x\|_{(i)}.$$

We now recall this fundamental result for unitarily invariant norms \cite[Theorem 2.1]{horn}:

\begin{thm}\label{normsupnui}
For every unitarily invariant norm $\|\cdot\|$ there is a subset $\A\subseteq \R^n_{+,\downarrow}$ such that
$$\|x\|=\max_{\alpha\in\A}\|x\|_\alpha.$$
\end{thm}

\begin{thm}\label{convdistancNUI}
Let $u(t)=e^{itx}e^{iy}=e^{ix(t)}\in U_n(\mathbb C)$ be a segment defined for $t\in [t_1,t_2]$. If $\|\cdot\|$ is an unitarily invariant norm on $\u(n)$ such that for $t\in [t_1,t_2]$ we have $\|x(t)\|_{\infty}<\frac{\pi}{2}$, then $t\mapsto\|x(t)\|$ is convex in $[t_1,t_2]$. 
\end{thm}

\begin{proof}
By Theorem \ref{normsupnui} we have 
$$\|x\|=\max_{\alpha\in\A}\|x\|_\alpha$$
for a subset $\A\subseteq \R^n_{+,\downarrow}$. If we show that $\|x(t)\|_\alpha$ is convex for any $\alpha\in\R^n_{+,\downarrow}$ then the conclusion follows, since $\|x(t)\|$ is a supremum of convex functions so it is convex. 

Let $\alpha\in\R^n_{+,\downarrow}$ and let $\sigma_1(t)\geq\sigma_2(t)\geq\dots\geq\sigma_n(t)$ be the eigenvalues of $x(t)$. The conditions of Theorem \ref{convsumdistgeneral} are satisfied and by this theorem 
$$\sum_{i=1}^m\sigma_i(t)$$
are convex for all $m$. Note that by setting $\alpha_{n+1}=0$ we have that
$$\|x(t)\|_{\alpha}=\sum_{i=1}^n(\alpha_i-\alpha_{i+1})\|x(t)\|_{(i)}=\sum_{i=1}^n(\alpha_i-\alpha_{i+1})\sum_{k=1}^i\sigma_k(t)$$
is a sum of convex functions, hence it is convex. This concludes the proof.  
\end{proof}
 
\begin{ex}[$p$-norms] For $p=[1,+\infty]$, we can consider the $p$-Schatten norms 
$$
\|x\|_p=(\tr|x|^p)^{1/p}=\left(\sum_{k=1}^n |\lambda_k(x)|^p\right)^{1/p}.
$$
In \cite{pschatten} we showed that the map $t\mapsto \|x(t)\|_p$ where $e^{ix(t)}=e^{itx}e^{iy}$ as before is convex for $p=2k$, $k\in\mathbb N$, and certain radius depending on $p$. With the previous theorem, we now have the result extended for all $p$, with a uniform, optimal radius for all $p$.
\end{ex}

\section{$\Ad$-invariant norms in compact semisimple groups}

In this section we recall some facts on $\Ad$-invariant norms on the Lie algebra $\k$ of a compact semisimple group $K$, and we prove a series of results that will enable us to extend the convexity results to $\Ad$-invariant norms in $\SU(n)$.

\begin{defn}[$\Ad$-invariant Finsler norms]
A \textit{Finsler norm} in $\mathfrak k$ is a map  $\|\cdot\|:\mathfrak k\to\mathbb R_{\ge 0}$ which is subadditive, non-degenerate and positively homogeneous: $\|\lambda v\|=\lambda\|v\|$ for any $\lambda\ge 0$ and any $v\in\mathfrak k$. An $\Ad$\textit{-invariant Finsler norm} is a Finsler norm invariant for the adjoint action of $K$: $\|\Ad_uv\|=\|v\|$ for any $v\in \mathfrak k$ and any $u\in K$.
\end{defn}

\begin{defn}
Let $\h$ be the Lie algebra of a maximal torus in a compact semisiple Lie group $K$, and let $\h^+$ be a positive Weyl chamber given by a choice of positive roots in the Cartan algebra $\h$. Let $\langle\cdot,\cdot\rangle$ denote the $\Ad$-invariant Killing form of $\k$, and for $x,\mu\in\k$ we define
$$\|x\|_\mu=\max_{k\in K} \langle x,\Ad_k\mu\rangle.$$
This is an $\Ad$-invariant Finsler norm in $\mathfrak{k}$, and it is a norm if and only if $\mu$ is symmetric, i.e. $\mu=-\mu$. See Section 2 in \cite{larmi} for all the details.
\end{defn}

\begin{ex}\label{ejsun}
When $K=\SU_n(\mathbb C)$, these where called \textit{orbit norms} in \cite{larey}. Let $\mu\in \k=\su(n)$ have eigenvalues $i\mu_1,\dots,i\mu_n$ with $\mu_1\geq\mu_2\geq \dots\geq\mu_n$ and let $x\in \su(n)$ have eigenvalues $ix_1,\dots,ix_n$ with $x_1\geq x_2\geq \dots\geq x_n$. Then it is not hard to check that
$$\|x\|_\mu=x_1\mu_1+\dots+x_n\mu_n.$$
See Lemma 2.12 in \cite{larey} for a proof and further details.
\end{ex}

\begin{thm}[Kostant]\label{kostant}
If $K$ is a compact Lie group and $T$ is a maximal torus in $K$ and $p:\k\to\h$ is the projection onto $\h$, then 
$$p(K.x )=\conv(\W .x),$$
where $\W .x$ is the Weyl group orbit of $x\in\k$.
\end{thm}

We recall the following corollary of Kostant's convexity theorem.

\begin{cor}\label{projball}
If $B\subseteq\k$ is an $\Ad$-invariant convex set then and $p:\k\to\h$ is the orthogonal projection onto $\h$ then 
$$p(B)=B\cap\h.$$
\end{cor}

\begin{proof}
Let $b\in B$ and let $k\in K$ such that $\Ad_k(b)\in \h$. By invariance of $B$ we know that $\Ad_k(b)\in B$. Hence by Kostant's convexity theorem
$$p(b)\in p(K.b)=\conv(\W.\Ad_k(b))\subseteq B\cap\h.$$
This proves that $p(B)\subseteq B\cap\h$. The other inclusion is trivial.  
\end{proof}

We recall from \cite[Lemma 2.9]{mozz} the following generalization of the rearrangement inequality 
\begin{equation}\label{reaineq}
\sum_{i=1}^nx_{\rho(i)}y_i\leq\sum_{i=1}^nx_iy_i
\end{equation}
for $x_1\leq\dots\leq x_n$, $y_1\leq\dots\leq y_n$ and all permutations $\rho$. From this one can deduce the following:

\begin{lem}\label{rearrengementineq}
Suppose $x,y\in \h$, then
$$\sup_{w\in\W}\langle x,w.y\rangle=\langle x,y\rangle$$
if and only if $x$ and $y$ belong to the same Weyl chamber. 
\end{lem}

Let $B_\mu=\{x\in\k:\|x\|_\mu\leq 1\}$ be the unit ball of $\k$ with the norm $\|\cdot\|_\mu$. 

\begin{lem}\label{hofercartan}
For $\mu\in \h$ we have 
$$B_\mu\cap\h=\{x\in\h:\max_{w\in\W}\langle x,w.\mu\rangle\leq 1\}.$$
\end{lem}

\begin{proof}
If $x,\mu\in\h$ then 
\begin{align*}
\|x\|_\mu&=\max_{k\in K} \langle x,\Ad_k\mu\rangle=\max_{k\in K} \langle x,p(\Ad_k\mu)\rangle\\
&=\max \langle x,\conv(\W.\mu)\rangle=\max_{w\in\W}\langle x,w.\mu\rangle
\end{align*}
where the third equality is given by Theorem \ref{kostant}. Hence the lemma follows.
\end{proof}

For a set $C\subseteq \k$ we denote with $C^\circ$ its polar dual and with $\bd C$ its boundary.

\begin{thm}\label{normsuphofer}
If $\|\cdot\|$ is an $\Ad$-invariant Finsler norm on $\k$ and $B\subset\k$ is its closed unit ball, then
$$\|x\|=\sup_{\mu\in \bd B^\circ \cap\h^+}\|x\|_\mu.$$
\end{thm}

\begin{proof}
This is equivalent to $B=\bigcap_{\mu\in \bd B^\circ \cap\h^+}B_\mu$, and by $\Ad$-invariance of the norms this is equivalent to 
$$
B\cap\h=\bigcap_{\mu\in \bd B^\circ \cap\h^+}B_\mu\cap\h.
$$
We denote $C=B\cap\h$ and $C_\mu=B_\mu\cap\h$, so have to show that $C=\bigcap_{\mu\in \bd B^\circ \cap\h^+}C_\mu$.

($\subseteq$) Since $C$ and $C_\mu$ are $\W$-invariant it is enough to prove that if $\mu\in B^\circ\cap\h^+$ then $C\cap\h^+\subseteq C_\mu\cap\h^+$. By Lemma \ref{hofercartan} and Lemma \ref{rearrengementineq} we have 
$$
C_\mu\cap\h^+=\{x\in\h^+:\langle x,\mu\rangle\leq 1\}
$$
Since $\mu\in B^\circ\cap\h^+\subset B^{\circ}\cap \h=C^{\circ}$, then  for $x\in C\cap\h^+\subset C$ we have $\langle \mu,x\rangle\leq 1$, that is $x\in C_\mu\cap\h^+$.

($\supseteq$) If $x\in\h^+$ and $x\notin C$ by the Hahn-Banach separation theorem there is a $\lambda\in\h$ such that $\langle x,\lambda\rangle >1$ and $\max_{c\in C}\langle c,\lambda\rangle=1$. By definition of the Weyl chamber there is a $w\in \W$ such that $w.\lambda\in\h^+$ and we set $\mu=w.\lambda$. Then since $x,\mu\in \h^+$ by inequality (\ref{reaineq}) we have 
$$
\langle x,\mu\rangle\geq \langle x,\lambda\rangle>1
$$
thus  $x\notin C_\mu$. We have to show that $\mu\in \bd B^\circ\cap\h^+$. By the $\W$-invariance of $C$ and of the inner product we have $\max\langle C,\mu\rangle =1$ and by Corollary \ref{projball} we have 
$$\max\langle B,\mu\rangle=\max\langle p(B),\mu\rangle=\max\langle C,\mu\rangle=1,$$
hence $\mu\in \bd B^\circ$.
\end{proof}

\subsection{The case of $\SU_n(\mathbb C)$}

We now prove the convexity results in the setting of this compact semisimple group.

\begin{thm}\label{convdistancAdinv}
Let $u(t)=e^{tx}e^{y}=e^{x(t)}\in \SU(n)$ be a segment defined for $t\in [t_1,t_2]$. If $\|\cdot\|$ is an $\Ad$-invariant norm on $\su(n)$ such that for $t\in [t_1,t_2]$ we have $\|x(t)\|_{\infty}<\frac{\pi}{2}$, then $t\mapsto\|x(t)\|$ is convex for $t\in [t_1,t_2]$. 
\end{thm}

\begin{proof}
By Theorem \ref{normsuphofer} we have 
$$\|x\|=\sup_{\mu\in \bd B^\circ \cap\h^+}\|x\|_\mu.$$ 
If we show that $\|x(t)\|_\mu$ is convex for any $\mu\in\su(n)$ then the conclusion follows, since $\|x(t)\|$ is a supremum of convex functions so it is convex. 

Let $\mu\in\su(n)$ have eigenvalues $i\mu_1,\dots,i\mu_n$ with $\mu_1\geq\mu_2\geq\dots\geq\mu_n$. Let $\theta_1(t)\geq\theta_2(t)\geq\dots\geq\theta_n(t)$ be the eigenvalues of $x(t)$. The conditions of Theorem \ref{convsumdistgeneral} are satisfied and by this theorem 
$t\mapsto \sum_{i=1}^m\theta_i(t)$ are convex functions for all $m$. Note that if we set $\mu_{n+1}=0$ then by Example \ref{ejsun} we have 
$$\|x(t)\|_{\mu}=\mu_1\theta_1(t)+\dots+\mu_{n}\theta_{n}(t)=\sum_{k=1}^n(\mu_{k}-\mu_{k+1})\sum_{j=1}^k\theta_j(t)$$
is a sum of convex functions, hence it is convex. This concludes the proof.  
\end{proof}

\subsection{Strict convextity}

\begin{thm}
Let $u(t)=e^{itx}e^{iy}=e^{ix(t)}$ be a path such that $\|e^{itx}e^{iy}-1\|< \sqrt{2}$, with eigenvalues $e^{i\theta_k(t)}$ and $\theta_1(t)\geq\theta_2(t)\geq\dots\geq\theta_n(t)$ for all $t\in [0,t_0]$. Let $\mu$ be a diagonal matrix with $\mu_1>\mu_2>\dots>\mu_n$, consider 
$$f_{\mu}(z(t))=\sup_{u\in U}\tr(u\mu u^{-1}z(t))=\mu_1\theta_1(t)+\mu_2\theta_2(t)+\dots+\mu_n\theta_n(t).$$
If there is a $t_*\in [a,b]$ such that all the eigenvalues of $u(t)$ are distinct and such that $f_{\mu}''(z(t_*))=0$ then $x$ and $y$ commute. 
\end{thm}

\begin{proof}
Set $\mu_{n+1}=0$ and note that $\sum_{j=1}^n\theta_j(t)$ is linear by Remark \ref{l1}. Then 
$$f_{\mu}(z(t))=\mu_1\theta_1(t)+\mu_2\theta_2(t)+\dots+\mu_n\theta_n(t)=\sum_{k=1}^n(\mu_{k}-\mu_{k+1})\sum_{j=1}^k\theta_j(t).$$
Let $t_*\in [a,b]$ such that all the eigenvalues of $u(t)$ are distinct and such that $f_{\mu}''(z(t_*))=0$. Then  
$$\sum_{k=1}^{n-1}(\mu_{k}-\mu_{k+1})\sum_{j=1}^k\theta''_j(t_*)=0$$
and since each term in the sum is non negative this implies that for all $k=1,\dots,n-1$ we have $\sum_{j=1}^k\theta''_j(t_*)=0$. By Proposition \ref{propconvsum}
\begin{equation*}
\sum_{k=1}^m\theta_{k}''(t_*)= 2 \sum_{k=1}^m\sum_{j=m+1}^n \frac{ \sin(\theta_{k}(t_*)-\theta_{j}(t_*))}{|e^{i\theta_{k}(t_*)}-e^{i\theta_{j}(t_*)}|^2}|\langle x v_{k}(t_*), v_{j}(t_*)\rangle|^2 
\end{equation*}
for $m=1,\dots,n-1$, where $(v_j(t_*))_{j=1}^n$ is an orthonormal basis of $u(t_*)$. Since $\|e^{it_*x}e^{iy}-1\|\leq r< \sqrt{2}$ we have
$$\frac{ \sin(\theta_{k}(t_*)-\theta_{j}(t_*))}{|e^{i\theta_{k}(t_*)}-e^{i\theta_{j}(t_*)}|^2}>0,$$
hence for all $j>k$ we have 
$$\langle x v_{k}(t_*), v_{j}(t_*)\rangle=0.$$
This implies that $x$ commutes with $u(t_*)=e^{it_*x}e^{iy}$, so it also commutes with $e^{-it_*x}e^{it_*x}e^{iy}=e^{iy}$. Finally, this implies that $x$ commutes with $-i \log(e^{iy})=y$.
\end{proof}

Before the next theorem, we need a couple of remarks:

\begin{rem}\label{lemconvexidad}
If $(f_n)_n$ be a sequence of $C^2$ convex functions from $[0,t_0]$ to $\R$ which converges uniformly to a function which is linear in some interval $[a,b]\subseteq [0,t_0]$. Then there is a subsequence $(f_{n_i})_i$ and a sequence $(t_i)_i$ such that $t_i\to t_*\in [a,b]$  and $f_{n_i}''(t_i)\to 0$. In fact, note that  $\liminf \left(\inf_{t\in [a,b]}f''(t)\right)=0$, otherwise the sequence $(f_n)_n$ cannot converge uniformly to a linear function in $[a,b]$.
\end{rem}

\begin{rem}[Hoffman-Wielandt inequality]Let $a,b$ be normal $n\times n$ matrices, denote $\lambda_k(a)$ and $\lambda_k(b)$ their respective eigenvalues. If $S_n$ is the permutation group of $n$ elements, then 
$$
\min_{\sigma\in\mathcal{S}_n}\sum_{k=1}^n\left|\lambda_k(a)-\lambda_{\sigma(k)}(b)\right|^2 \leq\left\Vert a-b\right\Vert^2.
$$
where $\|x\|_2=\sqrt{\tr(x^*x)}$ is the Frobenius (also called-Hilbert-Schmidt) norm of $x$. For a proof see for instance \cite[Theorem VI.4.1]{bhatia}.
\end{rem}

\begin{thm}\label{striconv}
Let $u(t)=e^{itx}e^{iy}=e^{iz(t)}$ be a curve such that $\|e^{itx}e^{iy}-1\|< \sqrt{2}$ for all $t\in [0,t_0]$, with eigenvalues $e^{i\theta_k(t)}$ with $\theta_1(t)\geq\theta_2(t)\geq\dots\geq\theta_n(t)$ for all $t\in [0,t_0]$. Let $\mu$ be a diagonal matrix with entries $\mu_1>\mu_2>\dots>\mu_n$. Then the function
$$f_{\mu}(z(t))=\sup_{u\in U}\tr(u\mu u^{-1}z(t))=\mu_1\theta_1(t)+\mu_2\theta_2(t)+\dots+\mu_n\theta_n(t)$$
is strictly convex if and only if $x$ and $y$ do not commute. If $x$ and $y$ commute then the function is piece wise linear. 
\end{thm}

\begin{proof} If $x$ and $y$ commute then $z(t)=tx+y$ and it is easy to check that $f_{\mu}(z(t))$ is piecewise linear.  So let  $\mu_{n+1}=0$ and note that $\sum_{j=1}^n\theta_j(t)$ is linear by Remark \ref{l1}. Then 
$$f_{\mu}(z(t))=\mu_1\theta_1(t)+\mu_2\theta_2(t)+\dots+\mu_n\theta_n(t)=\sum_{k=1}^n(\mu_{k}-\mu_{k+1})\sum_{j=1}^k\theta_j(t).$$
If $f_{\mu}(z(t))$ is not strictly convex, then because it is convex it must be linear in some interval $[a,b]\subseteq [0,t_0]$, so for each $k=1,\dots,n$ the sum $\sum_{j=1}^k\theta_j(t)$ must be linear for $t\in [a,b]$.

By Proposition \ref{aproxdistint} there is a sequence $(y_l)_l$ with $y_l\to y$ and such that for each $l\in \N$ the curves $u_l(t)=e^{itx}e^{iy_l}$ satisfy $\|e^{itx}e^{iy_l}-1\|\leq r< \sqrt{2}$ and has distinct eigenvalues 
$$\theta_{l,1}(t)>\theta_{l,2}(t)>\dots>\theta_{l,n}(t).$$ 
By the Hoffman-Wielandt inequality (see the previous remark) for each $k=1,\dots,n-1$ we have $\sum_{j=1}^k\theta_{l,j}(t)\to\sum_{j=1}^k\theta_j(t)$ uniformly, hence 
$$
f_l(t)=\sum_{k=1}^{n-1}(\mu_{k}-\mu_{k+1})\sum_{j=1}^k\theta_{l,j}(t)
$$
converges uniformly to a function which is linear in $[a,b]$. Therefore the previous remark implies that there is a subsequence of $(f_l)_l$, which we still denote by $(f_l)_l$, a sequence $t_l\to t_*$ such that $f_l''(t_l)\to 0$. Since for $m=1,\dots,n-1$ we have 
$$(\mu_{m}-\mu_{m+1})\sum_{j=1}^m\theta''_{l,j}(t_l)\leq f_l''(t_l)=\sum_{k=1}^{n-1}(\mu_{k}-\mu_{k+1})\sum_{j=1}^k\theta''_{l,j}(t_l)$$
we conclude that for $m=1,\dots,n-1$
$$\sum_{j=1}^m\theta''_{l,j}(t_l)\to_l 0.$$
Therefore, using formula (\ref{theta3}) of Proposition \ref{propconvsum} 
\begin{equation*}
\sum_{k=1}^m\theta_{l,k}''(t_l)= 2 \sum_{k=1}^m\sum_{j=m+1}^n \frac{ \sin(\theta_{l,k}(t_l)-\theta_{l,j}(t_l))}{|e^{i\theta_{l,k}(t_l)}-e^{i\theta_{l,j}(t_l)}|^2}|\langle x v_{l,k}(t_l), v_{l,j}(t_l)\rangle|^2 \to 0
\end{equation*}
for $m=1,\dots,n-1$. Since $\|e^{itx}e^{iy_l}-1\|\leq r< \sqrt{2}$, for $k>j$
$$\frac{ \sin(\theta_{l,k}(t)-\theta_{l,j}(t))}{|e^{i\theta_{l,k}(t)}-e^{i\theta_{l,j}(t)}|^2}\geq c>0.$$
The matrix $x$ is self-adjoint so we conclude that for all $j\neq k$
$$|\langle x v_{l,k}(t_l), v_{l,j}(t_l)\rangle|\to_l 0.$$

We now compute the Hilbert-Schmidt norm of $[x,u_l(t_l)]$ using the orthonormal basis $(v_{l,k}(t_l))_k$ of $u_l(t_l)$:
$$\|[x,u_l(t_l)]\|_2^2=\sum_{j,k=1}^n|e^{i\theta_{l,k}(t_l)}-e^{i\theta_{l,j}(t_l)}|^2|\langle xv_{l,k}(t_l),v_{l,j}(t_l)\rangle|^2.$$
Therefore $\|[x,u_l(t_l)]\|_2\to 0$ and $[x,u_l(t_l)]\to [x,u(t_*)]$ so that $[x,e^{it_*x}e^{iy}]=0$. This says that $x$ commutes with $u(t_*)=e^{it_*x}e^{iy}$, so it also commutes with $e^{-it_*x}e^{it_*x}e^{iy}=e^{iy}$. Finally, this implies that $x$ commutes with $-i \log(e^{iy})=y$ proving the main assertion of the theorem. 
\end{proof}

\begin{rem}[Repulsion of angles]  Note that  (notation as in the previous proof) we have
\begin{align*}
\sum_{k=1}^m\theta_{l,k}''(t_l) &= 2 \sum_{k=1}^m\sum_{j=m+1}^n \frac{ \sin(\theta_{l,k}(t_l)-\theta_{l,j}(t_l))}{|e^{i\theta_{l,k}(t_l)}-e^{i\theta_{l,j}(t_l)}|^2}|\langle x v_{l,k}(t_l), v_{l,j}(t_l)\rangle|^2 \\
&= 2\sum_{k=1}^{m}\sum_{j=m+1}^n \cot\left(\frac{\theta_{l,k}(t)-\theta_{l,j}(t)}{2}\right) |\langle x_l v_{k,l}(t),v_{j,l}(t) \rangle |^2
\end{align*}
where $\cot(\lambda)=\cos(\lambda)\sin(\lambda)^{-1}$ is the contangent function. All the terms are non-negative, but note also that as $\lambda\to 0^+$ (collapsing of different eigenvalues), we have $\cot(\lambda)\to +\infty$.  On the other hand, if the eigenvalues are far apart (i.e. near being opposite), these coefficients are near $0$.
\end{rem}

\section{Convexity of the geodesic distance}

In this section we prove that the distance from a point to a geodesic segment in $\SU_n(\mathbb C)$ is a convex map, for any  $\Ad$-invariant Finsler metric in $\SU_n(\mathbb C)$. We also state the result for invariant metrics in $\U_n(\mathbb C)$.

\begin{defn}\label{lengdist}
For a given $\Ad$-invariant Finsler norm in $\SU_n(\mathbb C)$, the \textit{length} of any smooth path $\gamma:[0,1]\to \SU_n(\mathbb C)$ is given by $L(\gamma)=\int_0^1 \|\gamma(t)^{-1}\gamma'(t)\|dt$. This length is bi-invariant:
$$
L(u\gamma)=L(\gamma u)=L(\gamma), \qquad u\in \SU_n(\mathbb C).
$$

The \textit{Finsler metric} $d(u,v)$ among $u,v\in \SU_n(\mathbb C)$ is defined as the infimum of the lenghts of piecewise smooth paths joining them in $\SU_n(\mathbb C)$. This is a standard distance which makes of $\SU_n(\mathbb C)$ a metric space except for the fact that we have $d(u,v)=d(v,u)$ for all $u,v$ if and only if the norm is fully homogeneous i.e. $\|\lambda v\|=|\lambda|\|v\|$ for any $\lambda\in\mathbb R$. The metric is also bi-invariant:
$$
d(upw,uqw)=d(p,q)\qquad \forall p,q,u,w\in \SU_n(\mathbb C).
$$

A path $\gamma$ in $\SU_n(\mathbb C)$ is a \textbf{geodesic} for the Finsler metric if $L(\gamma)=d$ where $d$ is the distance among the endpoints of $\gamma$.
\end{defn}

\begin{rem} The path $\gamma: [0,1]\ni  t\mapsto e^{tx}e^{y}$ is a geodesic for any $\Ad$-invariant Finsler metric in $\SU_n(\mathbb C)$  as long as $\|x\|_{\infty}\le \pi$, the later being the spectral norm of $x$. See \cite[Theorem 5.19]{larmi}.
\end{rem}

Therefore Theorem \ref{convdistancAdinv}  can be restated as follows:

\begin{thm}\label{convadsu} Let $\|e^{tx}e^{y}-1\|_{\infty}<\sqrt{2}$ in $[0,1]$, let $\|\cdot\|_{\phi}$ be any $\Ad$-invariant Finsler norm in $\su_n(\mathbb C)$, let $d_{\phi}$ be the induced Finsler metric in $\SU_n(\mathbb C)$. Then 
$$
t\mapsto d_{\phi}(1,e^{tx}e^{y})
$$
is convex for $t\in [0,1]$. 
\end{thm}

\begin{rem}[The convexity radius is optimal]\label{opt} The condition $\|e^{x(t)}-1\|_{\infty}<\sqrt{2}$ is equivalent to $\|x(t)\|_{\infty}<\pi/2$. If $n=2$  then $\SU_2(\mathbb C)\simeq S^3\subset \mathbb R^4$. If $\|\cdot\|_{\phi}$ is the Frobenius norm, the identification induces the standard Riemannian metric of the sphere $S^3$, and the spectral norm induces the $\ell_{\infty}$ norm of $\mathbb R^4$. But the Cartan algebra is one-dimensional, therefore all metrics are exactly equal there. It is known that the distance from the north pole to a geodesic segment $\gamma_t$ (part of a  great circle) is a convex function in any sphere $S^n$ only if this segment is inside the upper hemisphere. This amounts to the Euclidean linear distance from $\gamma_t$ to the north pole being less than $\sqrt{2}$, which in turn implies that the maximum of the coordinates is less than $\pi/2$.
\end{rem}

\begin{thm}[Strict convexity for orbit norms] If $\mu_1>\mu_2>\dots>\mu_n$,  let $\|\cdot\|_{\mu}$ be the orbit norm in $\su_n(\mathbb C)$, and let $d_{\mu}$ be the induced metric in $\SU_n(\mathbb C)$. If $\|e^{tx}e^y-1\|_{\infty}<\sqrt{2}$ in $[0,1]$ and $[x,y]\ne 0$ then $t\mapsto d_{\mu}(1,e^{tx}e^y)$ is strictly convex in $[0,1]$
\end{thm}
\begin{proof}
Immediate from Theorem \ref{striconv}.
\end{proof}

In a similar fashion, we obtain a convexity result for the distance in the setting of unitarily invariant norms of the full unitary group:

\begin{rem}[The case of $\U_n(\mathbb C)$] Let $\|\cdot\|$ be an unitarily invariant norm in $\u_n(\mathbb C)$, the rectifiable length and distance in $\U_n(\mathbb C)$ are defined accordingly (as in Definition \ref{lengdist}). Again the path  $\gamma: [0,1]\ni  t\mapsto e^{tx}e^{y}$ is a geodesic for any unitarily invariant norm as long as $\|x\|_{\infty}\le \pi$ (now see \cite[Theorem 16]{alv14} for a proof).
\end{rem}

Then Theorem \ref{convdistancNUI} can be restated as follows:

\begin{thm}\label{convrecu} Let $\|e^{tx}e^{y}-1\|_{\infty}<\sqrt{2}$ in $[0,1]$, let $\|\cdot\|_{\phi}$ be any unitariliy invariant norm in $\u_n(\mathbb C)$, let $d_{\phi}$ be the induced  metric in $\U_n(\mathbb C)$. Then 
$$
t\mapsto d_{\phi}(1,e^{tx}e^{y})
$$
is convex in $[0,1]$. 
\end{thm}

By Remark \ref{opt} this convexity radius is also optimal.

\section*{Acknowledgements}
This research was partially supported by CONICET and Universidad de Buenos Aires. Research Grants: UBACyT 2023, nº 20020220400256BA, FCEyN-UBA and ANPCyT  PICT2019-2019-04060.





\noindent
\end{document}